\documentclass[10pt]{amsart}
\pdfoutput=1

\usepackage{amssymb, amsfonts, amsmath, amsthm, amsbsy}
\ifx\luatexversion\undefined
\ifx\XeTeXversion\undefined
\usepackage[english]{babel}
\else
\usepackage{fontspec}
\usepackage{polyglossia}
\setmainlanguage{english}
\usepackage{microtype}
\fi
\else
\usepackage{fontspec}
\usepackage[english]{babel}
\usepackage[activate={true,nocompatibility},factor=1100,stretch=15,shrink=15,babel]{microtype}
\fi

\usepackage{todonotes, graphicx, hyperref, verbatim, enumerate}

\hypersetup{
    bookmarks=false,         
    pdftitle={},    
    pdfauthor={},     
    colorlinks=true,       
    linkcolor=black,          
    citecolor=black,        
    filecolor=black,      
    urlcolor=black           
}

\newcommand{\naturals}{\mathbb{N}}
\newcommand{\integers}{\mathbb{Z}}
\newcommand{\reals}{\mathbb{R}}

\newcommand{\torus}{\mathbb{T}}

\newcommand{\eps}{\varepsilon}

\newcommand{\deh}{\textup{d}}

\newcommand{\st}{\ \text{s.t.}\ }

\newcommand{\const}{C_\#}

\newcommand{\intr}[1]{\textup{int}\,#1}
\newcommand{\cl}[1]{\textup{cl}\,#1}

\newtheorem{thm}{Theorem}[section]
\newtheorem{lem}[thm]{Lemma}

\newtheorem{cor}[thm]{Corollary}
\theoremstyle{remark}
\newtheorem{rmk}[thm]{Remark}
\theoremstyle{definition}
\newtheorem{mydef}[thm]{Definition}

\newtheorem*{mthm}{Main Theorem}

\numberwithin{equation}{section}

\usepackage{verbatim}
\title[On dispersing billiards with corner points]{An expansion estimate for dispersing planar billiards with corner points}
\hypersetup{
    pdftitle={An expansion estimate for dispersing planar billiards with corner points},    
    pdfauthor={Jacopo De Simoi, Imre Péter Tóth},     
}
\author{Jacopo De Simoi}
\address{Jacopo De Simoi\\
  Dipartimento di Matematica\\
  II Universit\`{a} di Roma (Tor Vergata)\\
  Via della Ricerca Scientifica, 00133 Roma, Italy.}
\email{{\tt desimoi@mat.uniroma2.it}}

\author{Imre P\'eter T\'oth}
\address{Imre P\'eter T\'oth\\
  MTA-BME Stochastics Research group
  Egry J\'ozsef u. 1, H-1111 Budapest,  Hungary
}
\email{{\tt mogy@math.bme.hu}}

\input{bil.p}

\begin{document}
\begin{abstract}
  It is known that the dynamics of planar billiards satisfies strong mixing properties
  (e.g. exponential decay of correlations) provided that some expansion condition on
  unstable curves is satisfied.  This condition has been shown to always hold for smooth
  dispersing planar billiards, but it needed to be assumed separately in the case of
  dispersing planar billiards with corner points.

  We prove that this expansion condition holds for any dispersing planar billiard with
  corner points, no cusps and bounded horizon.
\end{abstract}

\maketitle
\section{Introduction}\label{s_introduction}
In the study of ergodic and statistical properties of hyperbolic dynamical systems with
singularities, it is essential to ensure a growth condition for manifolds tangent to the
unstable cone field (or \emph{u-manifolds}); in fact, even if hyperbolicity guarantees
their expansion on a local scale, singularities may cut u-manifolds in arbitrarily small
pieces, and this fact could effectively prevent us from obtaining any global result on the
system (see e.g. \cite{tsujii} for a concrete realization of this scenario).  Growth
conditions of this kind are usually stated in the form of a ``Growth Lemma'', which
ensures in precise terms that any sufficiently small u-manifold will be cut by
singularities in pieces that are, typically, large enough.

To fix ideas, consider a piecewise smooth map $\cm:\cs\to\cs$; we assume that $\cs$ is a
two-dimensional manifold and that $\cm$ has one-dimensional stable and unstable subspaces;
in systems under consideration, the singularity set is given by a union of smooth curves
of $\cs$; assume that we can find a uniform upper bound --as a function of $n$-- on the
number of smooth components of the singularity set of $\cm^n$ which join at any given
point $z\in\cs$: this is called a \emph{complexity bound} for the map.  Assuming some
uniform transversality condition between u-manifolds and singularity manifolds, a
complexity bound immediately implies a bound on the number of connected components of the
$n$-th image of a sufficiently small u-manifold.  Then it is possible to argue, by general
arguments, that a subexponential complexity bound implies the Growth Lemma (see
e.g. \cite{DimaLectures}).

When it is not possible, or not feasible, to obtain a complexity bound, one can look for a
more sophisticated condition, which contains more dynamical information: this strategy
relies on obtaining a so-called \emph{expansion estimate}: let $W\subset\cs$ be a small
u-manifold and let us denote by $W_i$ the connected components of its image $\cm W$, as
they are cut by singularities; let $\Lambda_i$ denote the minimum expansion rate of $\cm$
along $W$ in the corresponding preimage $\cm\inv W_i$; then we say $\cm$ satisfies a
one-step expansion estimate if
\begin{equation}
  \label{e_oneStepEstimate}
  \liminf_{\delta\to 0}\sup_{W:|W|<\delta} \sum_i\frac{1}{\Lambda_i}<1,
\end{equation}
where $|W|$ is the length of $W$ in a convenient norm.
Notice that a subexponential complexity bound for the map $\cm$ immediately implies an
expansion estimate for some iterate $\cm^n$ (or a $n$-step expansion estimate); therefore,
condition~\eqref{e_oneStepEstimate} is indeed weaker than a complexity bound.
Nevertheless, this condition is sufficient to prove the Growth Lemma by general arguments,
provided that we have (mild) control on the distortion of $\cm$ (see e.g. \cite[Theorem
5.52]{ChM}).

In this paper we obtain an expansion condition for planar billiards with corner points, no
cusps and bounded horizon.  Such condition is well known to hold (morally since
\cite{bsch1,bsch2}) for planar billiards with no corner points, but it needed to be
assumed separately --in a somewhat artificial fashion-- for billiard with corner points.
Our result implies that this additional assumption is unnecessary, and thus allows to
conclude that, for instance, the dynamics of planar billiards with corner points enjoys
exponential decay of correlation.

Let us remark that a subexponential complexity bound for finite horizon planar billiards
with corner points had been announced in \cite{Bun}, along with an outline of the proof.
In the present work, on the other hand, we directly obtain an expansion estimate, without
proving a complexity bound.  This approach has, in our opinion, two main advantages.
First, it does not seem to be possible to obtain a complexity bound sharper than
$\sim\exp(n-\log n)$, which is likely to be largely sub-optimal.  In our work we obtain an
expansion estimate which is more efficient than the one which would be obtained using such a
complexity bound.  Moreover --and in our opinion more importantly-- we believe that our
result can in principle be generalized to the case of unbounded horizon billiards.
\bigskip
\paragraph{\sc Acknowledgements} We would like to thank the ICTP - Trieste, where a large
portion of this work was done (during the 2012 ICTP-ESF workshop in Dynamical Systems).
This work was partially supported by the European Advanced Grant Macroscopic Laws and
Dynamical Systems (MALADY) (ERC AdG 246953) and by OTKA grant 71693.

Finally, we wish to thank D. Dolgopyat, for many valuable discussions and suggestions.
\section{Definitions}\label{s_definition}
In this section we provide all definitions and facts which are necessary for our
exposition; the reader might refer to \cite{bsch1,bsch2,ch,ChM} for (a wealth of)
additional details; our notations mostly follow, whenever possible, the ones used in the
given references.  We will state a number of lemmata about billiard dynamics, whose
proofs, unless better specified, can be found in the above references.

A \emph{billiard table} $\btable$ is the closure of a connected domain of $\reals^2$ (or
$\torus^2$) so that $\partial\btable$ is a \emph{finite} union of smooth curves
$\Gamma_i$, with pairwise disjoint interiors, that we call \emph{boundary curves} or
\emph{walls}.  Fix the standard orientation on $\partial\btable$ so that, walking along
the boundary in the positive direction, the interior of $\btable$ lies on the left hand
side.  If a point $x\in\partial\btable$ belongs to the interior of some $\Gamma_i$ we say
that $x$ is a \emph{regular point}, otherwise we call it a \emph{corner point}.  We assume
that every corner point $x$ is \emph{simple}, that is, any sufficiently small neighborhood
of $x$ intersects at most $2$ boundary curves\footnote{ This assumption is convenient,
  although inessential: our argument can be adapted to the non-simple case.}.  We assume
the billiard table to be \emph{dispersing}, that is, all $\Gamma_i$ to be outward convex
(i.e., given any two points $x, y$ on $\Gamma_i$, the interior of the segment joining
$x$ with $y$ does not intersect $\btable$): moreover, at any regular point
$x\in\partial\btable$ the curvature $\curv(x)$ is uniformly bounded away from zero.

We consider the dynamics of a point particle which moves with unit speed in the interior
of $\btable$ with elastic reflections on $\partial\btable$; we refer to this system as the
\emph{billiard flow}.  Notice that the flow is not well defined after a collision with a
corner point: we will resolve this issue later in this section.  Let
$\cs=\partial\btable\times[-\pi/2,\pi/2]$ denote the usual cross-section of the phase
space of the billiard flow; we employ standard coordinates $(r,\varphi)$ on $\cs$, where
$r$ is the arc-length parameterization of $\partial\btable$ and $\varphi\in[-\pi/2,\pi/2]$
is the angle between the inward normal to the boundary and the outgoing billiard
trajectory.  We denote by $|\cdot|$ the Euclidean metric in $\cs$.  In phase portraits, we
follow the convention of considering the $r$ coordinate as \emph{horizontal} and the
$\varphi$ coordinate as \emph{vertical} (see Figure~\ref{f_coordinates}).
\begin{figure}[!h]
  \centering
  \includegraphics{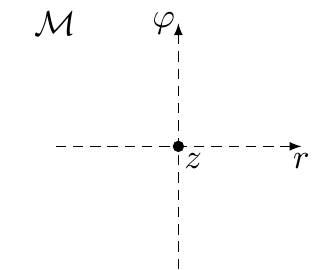}\hspace{2cm}
  \includegraphics{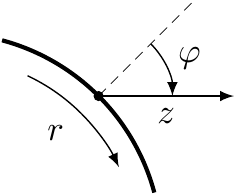}
  \caption{Sketch of phase space coordinates on the billiard table: the phase point $z$
    corresponds to the trajectory leaving $\partial\btable$ at the point identified by the
  boundary coordinate $r$ with angle $\varphi$ from the normal pointing inside the
  billiard table.}
  \label{f_coordinates}
\end{figure}

Let $\cm:\cs\to\cs$ be the Poincar\'e map of the
billiard flow, which is commonly called \emph{billiard map}.  For $z\in\cs$, denote by
$\tau(z)$ the return time of the flow to $\cs$.  The map $\cm$ preserves the smooth
probability measure $\deh\nu = C_\nu \cos\varphi\deh r\deh\varphi$, where $C_\nu$ is a
normalization constant.

We can write $\cs=\bigcup_i\cs_i=\bigcup_i\Gamma_i\times[-\pi/2,\pi/2]$; each $\cs_i$ is
thus diffeomorphic to either a cylinder (if $\Gamma_i$ is a closed curve surrounding a
scatterer with no corner points, i.e. $\partial\Gamma_i=\emptyset$) or a square
(otherwise).  Let us define
\begin{align*}
  \grazing_0&=\bigcup_i\grazing_{0,i}=\bigcup_i\Gamma_i\times\{\pm \pi/2\}&
  \corner_0&=\bigcup_i\corner_{0,i}=\bigcup_i\partial\Gamma_i\times[-\pi/2,\pi/2].
\end{align*}
The set $\grazing_0$ contains all grazing collisions, whereas $\corner_0$ contains all
collisions with a corner point; let $\sing_0=\grazing_0\cup\corner_0$.  The maps $\cm$ and
$\cm\inv$ are piecewise smooth: their singularities lie on the set $\cm\inv\sing_0$ and
$\cm\sing_0$, respectively. We denote by $\sing_l=\cm^l\sing_0$ and
$\sing_{m,n}=\bigcup_{l=m}^{n}\sing_l$ so that the singularities of $\cm^n$ and $\cm^{-n}$
lie, respectively, on the sets $\sing_{-n,0}$ and $\sing_{0,n}$.  Similarly, we define
$\grazing_l$, $\grazing_{m,n}$ and $\corner_l$, $\corner_{m,n}$.
\begin{rmk}\label{r_multi_step_sing_defined}
  Note that our definition of the sets $\sing_l$ is not really precise at the moment.
  Indeed, we have not yet fixed the definition of $\cm$ at corner points, and we definitely
  need to do that for $\sing_l$ to make sense.  We will fill this gap later, making sure that
  $\sing_{-n,0}$ indeed becomes the singularity set for $\cm^n$.
\end{rmk}
If $\partial\btable$ has no corner points\footnote{ Note that dispersing billiards without
  corner points are only realizable as subsets of $\torus^2$}, then $\cs$ is given by the
disjoint union of smooth simple closed curves, and $\corner_0=\emptyset$.  On the other
hand, if $\partial\btable$ has corner points, then some of the $\cs_i$ will be rectangles
which share a portion of both vertical edges with some other rectangles in $\{\cs_j\}$.
It is instead convenient to cut the phase space along corner points and redefine\footnote{
  The reader will excuse our abuse of notation} $\cs$ as the \emph{disjoint} union
$\cs=\bigsqcup_i\cs_i$, in such a way that each $\cs_i$ is a connected component of $\cs$:
this allows us to write $\partial\cs_i=\grazing_{0,i}\cup\corner_{0,i}$ and
$\partial\cs=\grazing_0\cup\corner_0$.  More importantly, for each $i$,
$\cm|_{\cm\inv\cs_i}$ uniquely extends to the closure $\overline{\cm\inv\cs_i}$ by
continuation:  the map $\cm$ can thus be multi-valued at points belonging to the closure
of several $\cm\inv\cs_i$'s.

We now introduce the notion of \emph{proper} and \emph{improper} collisions of a billiard
trajectory with a boundary point $x\in\partial\btable$: intuitively, improper collisions
are such that they can be avoided by perturbing the trajectory; for smooth billiard
tables, only tangential collisions can be improper, and so there is no need to introduce a
separate notion.  In our case, we first need to introduce a few auxiliary definitions.
Let $x\in\partial\Omega$; denote with $w_-$ (resp. $w_+$) the limit of tangent vectors to
the boundary from the left (resp. right) according to the orientation chosen in at the
beginning of this section (they can coincide
if, for instance, $x$ is regular).  Then $w_+$ and $-w_-$ cut the tangent space (that is
$\reals^2$) in two open sectors: we call \emph{internal} the sector bounded (going
clockwise) by $-w_-$ and $w_+$ and \emph{external} the opposite sector.  We denote with
$\gamma$ the angle of the internal sector; if $x$ is regular, then necessarily
$\gamma=\pi$; otherwise a corner point is said to be \emph{acute} if $0<\gamma<\pi$,
\emph{flat} if $\gamma=\pi$ and \emph{obtuse} if $\gamma>\pi$.  If $\gamma=0$ we have a
\emph{cusp}; we assume our billiard table to have no cusps.
\begin{mydef}
  We say that a collision is \emph{proper} if the incoming velocity vector lies in the
  external sector and \emph{improper} otherwise.
\end{mydef}
Notice that if a corner point $x$ is acute, then all collisions hitting $x$ are proper.
We assume the \emph{bounded horizon} condition, that is, that the maximal length of a
straight billiard trajectory (i.e. between two proper collisions) is uniformly bounded
above: $\tau(z)\leq\tmax<\infty$.  Let us summarize our
\begin{assumptions}
  We assume that our billiard table $\btable$ is such that:
  \begin{itemize}
  \item[(A0)] all corner points are simple;
  \item[(A1)] all corner points have strictly positive internal angle, i.e. there are no cusps;
  \item[(A2)] the bounded horizon condition is satisfied.
  \end{itemize}
\end{assumptions}

It is well known that the billiard flow admits a unique (non smooth) continuation after a
grazing collision; on the other hand, as we mentioned before, the trajectory is not --in
general-- well defined after a collision with a corner point.  Since we are interested in
statistical properties of the smooth invariant measure $\nu$ --or, more generally, in
describing orbits of Lebesgue-typical phase points-- we could in principle define the
dynamics arbitrarily for the zero measure set of phase points whose trajectories hit a
corner point, or even leave it undefined: the choice of definition does not influence the
statistical properties of the system.  However, understanding the possible trajectories
occurring \emph{near} these singular points is of utmost importance in our study, so it is
convenient to define the dynamics at (corner) singularities so that it reflects --at least
to some extent-- these possibilities.  In particular, it is convenient to make sure that
singularities of higher iterates of $F$ (which can naturally be defined as the boundaries
of the domains of smoothness consisting of regularly colliding phase points) can be
obtained as (inverse) images by $F$ of the singularities (recall
Remark~\ref{r_multi_step_sing_defined}).  This is why we define the dynamics $F$ at corner
points as a \emph{possibly multi-valued} function, having \emph{branches} corresponding to
all possible limits of nearby trajectories.

Let us consider a trajectory having a collision with a corner point $x\in\partial\btable$;
we call this the \emph{reference trajectory}.  If the collision is improper, we have three
possibilities: trajectories close to the reference one may hit either one of the two walls
which join at $x$, or miss both of them.  However, it is easy to see (see
e.g. Figure~\ref{f_properCorner}) that if there exist nearby trajectories colliding with
the back wall, then the latter is necessarily tangent to the reference trajectory, and
thus the corresponding continuation coincides with the one relative to trajectories
missing the walls.

In case of a proper collision, there is no possibility of a nearby trajectory missing both
walls; by arguments similar to the ones above we can prove the following
\begin{lem}[{see \cite[Section 2.8]{ChM}}]
  Let $x\in\partial\btable$ be a corner point of interior angle $\gamma$; a trajectory of
  the billiard flow hitting $x$ has at most $2$ possible continuations.  Every trajectory
  colliding with $x$ admits a unique continuation if and only if $\gamma=\pi/n$, for some
  natural number $n$.
\end{lem}
We emphasize that the above lemma refers to possible continuations of the \emph{flow};
indeed, the continuation of the map is also two-valued, but the ``unique if
$\gamma=\pi/n$'' part of the statement is not true for the map.

After a single collision with a wall occurring at a corner point, the trajectory might
immediately collide with the adjacent wall; in this case we say that we have an
\emph{immediate collision}; if, in addition, the collision is grazing, we say we have an
\emph{immediate grazing collision}.  In principle, a trajectory may undergo several
subsequent immediate collisions before leaving the corner; this is usually referred to as
a \emph{corner sequence}\footnote{ Indeed, in the literature, one generally refers to a
  \emph{corner sequence} as a set of consecutive collisions occurring in a neighborhood of
  a corner point rather than at the corner point.}.  Assumption (A1) implies that the
number of immediate collisions in a corner sequence is uniformly bounded above for any
given billiard table (see e.g. \cite{ChM}).  Notice that an immediate grazing collision is
necessarily the last one in a collision sequence (see also \cite[Section 9]{ch}).
\begin{figure}[!ht]
  \centering
  \includegraphics{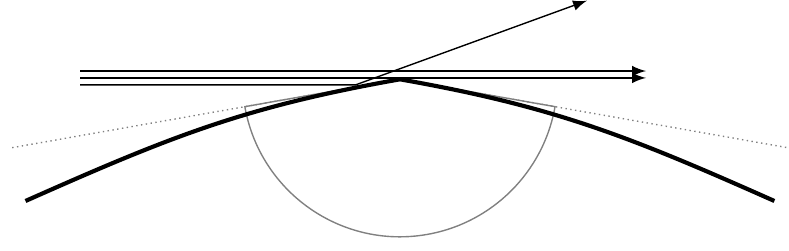}
  \caption{Trajectories in a neighborhood of a reference trajectory undergoing improper
    collision: some trajectories miss the scatterer and fly by, others will hit the
    scatterer frontally with some angle.  No collision with the back of the scatterer is
    possible unless the back wall is tangent to the reference trajectory.}
  \label{f_properCorner}
\end{figure}

One of the peculiar technical difficulties in the theory of dispersing billiards is the
fact that derivatives blow up at grazing singularities: this fact prevents simple
distortion control, which is crucial to obtain ergodic and statistical results.  In order
to provide an elegant solution to this problem, Sinai (see \cite{bsch2}) introduced
so-called \emph{homogeneity strips} $\sst_{\pm k}\subset\cs$ defined as follows:
\begin{align*}
  \sst_k&=\{(r,\varphi)\in\cs\st\varphi\in(\pi/2-k^{-2},\pi/2-(k+1)^{-2}]\}\\
  \sst_{-k}&=\{(r,\varphi)\in\cs\st\varphi\in[-\pi/2+(k+1)^{-2},-\pi/2+k^{-2})\},
\end{align*}
for $k\geq k_0$ where $k_0>0$ is fixed large enough (it will be specified below, after the
statement of Lemma~\ref{l_grazing}); for notational convenience, let us also define
\[
\sst_0=\{(r,\varphi)\in\cs\st\varphi\in[-\pi/2+k_0^{-2},\pi/2-k_0^{-2}]\}.
\]
Then on each $\sst_k$, the derivatives of $\cm$ will be roughly comparable and we will be
able to control its distortion.
We let $\exgrazing_0=\grazing_0\cup\bigcup_k\partial\sst_k$ and $\exsing=\exgrazing_0\cup
V_0$.  The boundaries $\partial\sst_k$ of homogeneity strips are called \emph{secondary
  singularities}. Our goal is now to prove the one-step expansion
estimate~\eqref{e_oneStepEstimate} for the billiard system in which these (infinitely
many) secondary singularities are also taken into account.  As usual in the case of planar
billiards, this requires some extra technical work, but means no real extra difficulty,
because expansion in the homogeneity strips with $k\ge k_0$ is huge and their total
contribution to the sum in \eqref{e_oneStepEstimate} can be made arbitrarily small by
choosing $k_0$ large enough (this will in fact be the content of Lemma~\ref{l_grazing}).
\begin{lem}[Invariant cones]
  At each point $(r,\varphi)$ let us consider nonzero vectors $(\delta r,\delta\varphi)$ belonging
  to the tangent space $T\cs=\reals^2$.  Let us call \emph{increasing cone} the cone given by
  $\{\delta r\delta\varphi\geq 0\}$ and \emph{decreasing cone} the cone $\{\delta
  r\delta\varphi\leq 0\}$. Then the differential $\deh\cm$ maps the increasing cone
  strictly into itself and likewise $\deh\cm\inv$ maps the decreasing cone strictly into
  itself.
\end{lem}
The push-forward (resp. pull-back) of the increasing cone (resp. decreasing cone) by $\cm$
defines a cone field, that is called \emph{unstable cone} (resp. \emph{stable cone}) and
denoted with $\cone^\unstable$ (resp. $\cone^\stable$).  A vector is said to be unstable
(resp. stable) if it belongs to the unstable (resp. stable) cone; likewise a smooth curve
$W\subset\cs$ is said to be \emph{unstable} or a \emph{u-curve} (resp. \emph{stable} or a
\emph{s-curve}) if the tangent vector at any point of $W$ is unstable (resp. stable).
We now collect a few known important results about dispersing billiards with
corner points.  We closely follow the exposition of \cite[Section 9]{ch}.
\begin{lem}[Transversality]\label{l_transversality}
  For any $z\in\cs$, the angle between stable and unstable cones at $z$ is uniformly
  bounded away from zero.
\end{lem}
\begin{rmk}
  If the billiard table $\btable$ has no corner points, a stronger version of the above
  lemma holds, that is, for any $z,z'\in\cs$, the angle between the stable cone at $z$ and
  the unstable cone at $z'$ is uniformly bounded away from zero.  In any case the
  transversality condition stated in Lemma~\ref{l_transversality} suffices for our
  purposes.
\end{rmk}
\begin{lem}[Structure of singularities]\label{l_alignment}
  For any $l\in\integers$, the set $\sing_l$ is a finite union of $C^3$ smooth curves;
  if $l>0$, then such curves are unstable, otherwise if $l<0$ they are stable.  A point
  $z\in\sing_{n,m}$ is said to be \emph{simple} if a sufficiently small neighborhood of
  $z$ intersects non-trivially only one smooth curve of $\sing_{n,m}$ and \emph{multiple}
  otherwise. For fixed $n<m$ there exist only finitely many multiple points.
\end{lem}
\begin{lem}[Expansion of unstable vectors (see {\cite[Lemma~9.1]{ch}})]
  Let $z=(r,\varphi)\in\cs$ and $\cm z=(r',\varphi')\in\cs$.  There exists a constant
  $\Const>0$, which depends on the billiard table only, so that, for any
  $v\in\cone^\unstable(z)$:
  \[
  |\deh\cm v|\geq \frac{\Const}{\cos\varphi'}|v|
  \]
\end{lem}
In particular the above lemma ensures that, if $\cm z\in\sst_k$, then the expansion rate
along unstable vectors at $z$ is bounded below by $\const k^{2}$.  This fact will be
crucial for the proof of Lemma~\ref{l_grazing}.  Notice that, even if the expansion rate
of unstable vectors diverges as $\varphi\to\pm\pi/2$, this divergence is in fact integrable; more
precisely we can show that
\begin{lem}[Maximal expansion of u-curves (see
  {\cite[Exercise~4.50]{ChM}})]\label{l_maximalExpansion}
  There exists a constant $\constMaxLen$, which depends on $\btable$ only, so that for any
  u-curve $W$ and any connected component $W'\subset \cm W$, we have:
\[
|W'|\leq \constMaxLen |W|^{1/2}.
\]
\end{lem}
\begin{lem}[Hyperbolicity, see {\cite[Lemma~9.2]{ch}}]\label{l_hyperbolicity} The map $\cm$ is uniformly hyperbolic
  in the Euclidean metric, that is, there exist $c>0$ and $\minLambda>1$ so that for any
  $v\in\cone^\unstable$ we have
  \begin{equation*}
    |\deh\cm^n(v)|\geq c\inv\minLambda^n.
  \end{equation*}
  A similar statement holds for stable vectors: for any $v\in\cone^\stable$ we have
  $|\deh\cm^{-n}(v)|\geq c\inv\minLambda^n$.
\end{lem}
Consider a u-curve $W$: the image of $W$ by $F$ is given by the union of a finite number of
connected components, since $W$ is cut by singularities $\sing$; each component
might be further subdivided by singularities $\exsing$ into countably many pieces,
which we call \emph{H-components}.  We denote by $\{W\hc i\}$ the H-components of $\cm W$
and by $\{W\hcn{i}{n}\}$ the H-components of $\cm^n W$.
\begin{mydef}
  If $W\hc i\subset\mathbb{H}_0$, we say that $W\hc i$ is \emph{regular}; otherwise we
  call $W\hc i$ \emph{nearly grazing}.  Likewise, we say that $W\hcn i n$ is
  \emph{regular} if for any $0\leq l<n$ we have that $\cm^{-l}W\hcn i n\subset\sst_0$ and
  \emph{nearly grazing} otherwise; if $W\hcn i n$ is nearly grazing, define
  \[
  \text{rank}(W\hcn i n)=\min\{p\in\{1,\cdots,n\}\st \cm^{-(n-p)}W\hcn i n\cap\sst_0=\emptyset\}.
  \]
  For $n>0$, define the \emph{regular $n$-complexity of $W$}, denoted with $\compl_n(W)$,
  as the number of regular H-components of $\cm^n W$; if $n=0$ we set conventionally
  $\compl_0(W)=1$.  Finally, define:
  \[
  \compl_n=
  \liminf_{\delta\to0}\sup_{W:\,|W|\leq\delta}\compl_\Ncorner(W).
  \]
\end{mydef}
\begin{lem}\label{l_corner}
  There exists $\Ncorner\in\naturals$, which depends only on $\btable$, so that
  \begin{equation}\label{e_complexityAtN}
    \compl_\Ncorner<\frac{1}{3}c\inv\minLambda^\Ncorner,
  \end{equation}
  where $c$ and $\Lambda_*$ are the ones obtained by Lemma~\ref{l_hyperbolicity}.
\end{lem}
\begin{lem}\label{l_grazing}
  For any $\eps>0$ we can choose $k_0$ large enough in the definition of homogeneity
  strips so that
  \begin{equation}\label{e_grazing}
    \liminf_{\delta\to 0} \sup_{W:\,|W|\leq\delta}{\sum_i}^*\frac{1}{\Lambda_i}<\eps,
  \end{equation}
  where ${\sum}^*$ denotes that the sum is restricted to nearly grazing components.
\end{lem}
We emphasize that Lemma~\ref{l_grazing} above is stated for a single iteration of $F$;
this will be sufficient for our purposes.  We fix $k_0$ so that Lemma~\ref{l_grazing}
holds for $\eps=\frac12c\Ncorner\inv\minLambda^{-2\Ncorner}$, where $\Ncorner$ is the one
provided in Lemma~\ref{l_corner} and $c$ and $\Lambda_*$ are the ones obtained by
Lemma~\ref{l_hyperbolicity}; the reader will find the reason for this choice in the proof
of our Main Theorem.

We will prove Lemmata~\ref{l_corner} and~\ref{l_grazing} in the
next section.  Given for granted the above two statements, we can now state and prove our
\begin{mthm}
  Let $\btable$ satisfy our Standing Assumptions (A0-A2), then
  \begin{equation}
    \label{e_expansionEstimate}
    \liminf_{\delta\to0}\sup_{W:\,|W|\leq\delta}\sum_{i}\frac{1}{\Lambda_{i,N}} < 1,
  \end{equation}
  where $N$ is the one obtained by Lemma~\ref{l_corner} and we denote by $\Lambda_{i,N}$
  the minimum expansion of $\cm^N$ on $\cm^{-N}W\hcn i N$.
\end{mthm}
\begin{proof}
  For any $n>0$ and u-curve $W$ define
  \begin{align*}
    \bdex_n(W) &= \sum_i\frac1{\Lambda_{i,n}}&
    \bdex_n(\delta) &= \sup_{W:|W|\leq\delta}\bdex_n(W);
  \end{align*}
  and set $\bdex_0(\cdot)=1$ by convention.  By Lemmata~\ref{l_corner} and~\ref{l_grazing} and
  our choice of $k_0$ made above, we know that there exists a $\delta_0$ so that
  \begin{align*}
    \sup_{W:\,|W|\leq\delta_0}\compl_\Ncorner(W)&<\frac{1}{2}c\inv\minLambda^\Ncorner&
    \sup_{W:\,|W|\leq\delta_0}{\sum_i}^*\frac{1}{\Lambda_i}&<c\Ncorner\inv\minLambda^{-2\Ncorner}.
  \end{align*}
  Recall that Lemma~\ref{l_maximalExpansion} gives an a priori bound on the length of
  $H$-components: $|W^n_i|\leq |\cm^n W|\leq \constMaxLen^n|W|^{2^{-n}}$; let us define
  $\delta_n=(\delta_0\constMaxLen^{-n})^{2^n}$.  We claim that for any u-curve $W$ with
  $|W|<\delta_n$, the following estimate holds:
  \begin{equation}\label{e_tree}
    \bdex_n(W) \leq \compl_n(W)\cdot c\minLambda^{-n} + c\Ncorner\inv\minLambda^{-2\Ncorner}\sum_{r=1}^{n} \compl_{r-1}(W)\bdex_{n-r}(\delta_{n-r}).
  \end{equation}
  In fact, a H-component $W\hcn i n$ can either be regular, or not.  By definition, the
  number of regular H-components is $\compl_n(W)$: hence their contribution to $\bdex_n$
  is bounded by the first term in~\eqref{e_tree}.  On the other hand, the contribution of
  all nearly grazing H-components of rank $r$ is bounded by the $r$-th term in the sum
  in~\eqref{e_tree} using Lemma~\ref{l_grazing}.  In particular, since $\minLambda>1$ and
  $\compl_{r-1}(W)\leq \compl_{n}(W)$ we have:
  \[
  \bdex_n(W)\leq \compl_n(W)c\left[1+\Ncorner\inv\minLambda^{-2\Ncorner}\sum_{r=1}^{n}\bdex_{n-r}(\delta_{n-r})\right].
  \]
  We can thus ensure, by induction, that $\bdex_n(\delta_n)<\minLambda^n$ for all $0\leq n
  < \Ncorner$; hence we use~\eqref{e_tree} one final time to obtain
  \[
  \bdex_\Ncorner(W)\leq 2\compl_\Ncorner(W) c\minLambda^{-\Ncorner} <1,
  \]
  which concludes the proof.
\end{proof}
\begin{cor}
  Let $\btable$ satisfy our Standing Assumptions (A0-A2), then the billiard map features
  exponential decay of correlations and the central limit theorem for H\"older-continuous
  observables.
\end{cor}

\begin{proof}
  The statement is proven in \cite{ch} under the assumption that the complexity is
  sub-exponential, but if fact only the statement of our Main Theorem is used.
\end{proof}

As further corollaries, many other strong statistical properties are satisfied by the
billiard map under our assumptions.  In fact, in \cite{ch} a \emph{Young tower}
(introduced in the seminal work of Young \cite{Y2}) with an exponential tail of the return
times is constructed. For such ``Young systems'', many further statistical properties have
been proved, including large deviations (\cite{RY}, \cite{MN2}), local limit laws
(\cite{SV}), almost sure invariance principles (\cite{MN1}), and Berry-Ess\'een type
theorems (\cite{P}).  It is worthwhile to mention that the same strong statistical
properties can also be obtained by means of the more geometrical \emph{coupling} approach
(introduced in \cite{yo99} and further developed in \cite{cd,ch06b}); the reader can find
a detailed exposition of the application of this technique to billiard systems in
\cite[Section 7]{ChM}.

\section{Proof of main technical lemmata}
In smooth dispersing billiards, the free path $\tau(r,\varphi)$ is always bounded away
from zero; however, in our situation, $\tau(r,\varphi)$ can become arbitrarily small if
$r$ approaches an acute corner point, where corner series may occur.  On the other hand,
for any fixed small $\varrho>0$, there exists $\tau_*(\varrho)>0$ so that
$\tau(r,\varphi)\geq\tau_*(\varrho)$ if $r$ does not belong to the $\varrho$-neighborhood
of any acute corner point; in particular
\begin{rmk}\label{r_freePathImproper}
  The free path between two improper collisions is uniformly bounded away from zero by
  some $\tau_*>0$.
\end{rmk}
In order to prove Lemma~\ref{l_corner} we study how singularity curves can join at a
multiple point. For $\rho>0$, and $z\in\cs$, let us denote by $\nhbd_{z,\rho}$ the open
$\rho$-ball around $z$.  By the facts stated in the previous section, we obtain
\begin{lem}[Local singularity portrait (see also {\cite[Theorem~6.1]{Bun} or
    \cite[Lemma~8.6]{bsch1}})]\label{l_portrait}$\,$\\
  There exists a strictly positive (non-increasing) sequence $\{\rho_n\}$ so that for any
  $n>0$:
  \begin{enumerate}
  \item for any $z,z'$ distinct multiple points of $\sing_{-n,0}$ we have
    $\cl{\nhbd_{z,\rho_n}}\cap \cl{\nhbd_{z',\rho_n}}=\emptyset$;
  \item fix a multiple point $z$: the neighborhood $\nhbd=\nhbd_{z,\rho_n}$ is cut by $\sing_{-n,0}$
    in a finite\footnote{ Recall that the set $R$ does not contain the boundaries of
      homogeneity strips} number of sectors, which we call sectors \emph{of order $n$} and denote
    with $\{\sect_{i}\}_{i\in\{1,\cdots,k(n,z)\}}$; then for each $i$ the map $\cm^n|\sect_{i}$ is
    smooth.  By Lemma~\ref{l_alignment} we obtain that each sector is bounded by stable
    curves (see Figure~\ref{f_sectorsU}).
  \end{enumerate}
\end{lem}
\begin{figure}[!h]
  \centering
  \includegraphics{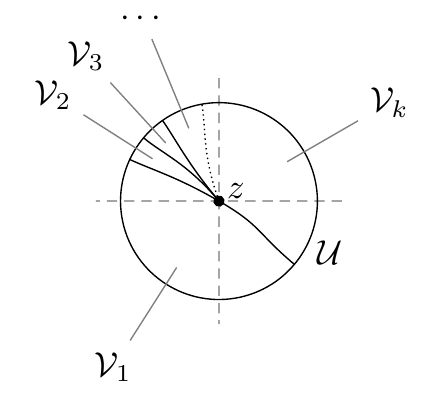}
  \caption{Singularity portrait around a point $z\in\intr\cs$.}
  \label{f_sectorsU}
\end{figure}
The description given by the above lemma is valid also for $z\in\partial\cs$, with the
difference that, in this case, the neighborhood $\nhbd$ can be either a half-ball or a
quarter-ball, depending on $z$.  In particular, if $z\in\corner_0$, we could have an
immediate collision; however, no singularity curves other than $\sing_0$ can join at $z$
in this case.  On the other hand, any neighborhood of an immediate grazing collision will
necessarily contain a curve belonging to $\sing_{-1,0}$ which joins $\corner_0$
tangentially.

Observe that the increasing quadrants (i.e. the North-East (NE) and the South-West (SW)
quadrant) cannot be cut by any future singularity; we call them \emph{inactive
  quadrants}. On the other hand, both decreasing quadrants (i.e. NW and SE) might be cut
by a future singularity; we call them \emph{active quadrants}.  Let us fix $\hcs_0$ to be
the closure of a neighborhood of $\sst_0$; for definiteness we let
$\hcs_0=\sst_{-k_0}\cup\sst_{0}\cup\sst_{k_0}$.  We say that a sector $\sect_i$ is
\emph{regular} if
\[
\lim_{\sect_i\ni z'\to
  z}F^lz'\in\hcs_0\ \text{for all}\ 0 < l \leq n;
\]
otherwise we say that $\sect_i$ is \emph{nearly grazing}.
Let us denote with $\cplex_n(z)$ the number of regular sectors of $\nhbd_{z,\rho_n}$ (meeting
at $z$); notice that $\cplex_n(z)$ makes sense also if $z$ is a simple point of
$\sing_{-n,0}$, and it can be at most $2$.  We introduce the notation
$\cplex_n=\sup_{z\in\cs}\cplex_n(z)$.

In Section~\ref{s_introduction} we introduced the \emph{complexity} of the singularity set
of $\cm^n$ as the ``number of smooth components of the singularity set of $\cm^n$ which
join at a given point $z$''.  However, from what we did so far, it is clear that it is not
really the number of singularities that matters, but rather the ``number of domains of
smoothness of $\cm^n$'' which join at $z$ (in fact, even defining the dynamics at singular
points is just an auxiliary tool to count these).  So, in the sequel, when considering
complexity growth, we will always think of these ``possible trajectories of non-singular
phase points near $z$'', and not the singularity set itself.

One of the key ideas in our present approach to studying growth of u-curves is that we
make advantage of the strong expansion occuring at nearly grazing collisions.  As a
result, we do not need to count every component (near some $z\in\cs$), into which the
phase space is cut by the singularities of $\cm^n$, but we can consider only those which
never experience such a strong expansion.  This is the content of the following lemma.
\begin{lem}\label{l_linearBound}
  There exists $\Xi$ depending only on $\btable$ so that for any $n>0$
  \[
  \cplex_n\leq\Xi n.
  \]
\end{lem}

As we will mention later, in presence of corner points, such a linear bound is not true
for the total complexity of the singularity set. Indeed, our ``regular complexity'' turns
out to be \emph{much smaller}, and especially much easier to control, than the total
complexity.  This is true despite the fact that corner points, which are responsible for
branching of the trajectories, do not themselves cause trouble.  In our detailed
study of the mechanism of complexity growth, we will see that it is only corner points and
grazing collisions \emph{together} that make a superlinear growth possible.

Assuming the above lemma, the proof of Lemma~\ref{l_corner} follows from a variation on
rather standard arguments.
\begin{proof}[Proof of Lemma~\ref{l_corner}]
  Let us fix $\Ncorner$ so that $\Xi\Ncorner<\frac{1}{3}c\inv\minLambda^{\Ncorner}$;
  notice that since $\Xi$, $c$ and $\minLambda$ depend only on $\btable$, then so does
  $\Ncorner$.  Let $h=k_0^{-2}-{k_0+1}^{-2}$ denote the height of the homogeneity strip
  $\sst_{k_0}$.  We will show that there exists a $\delta$ so that
  \[
  \sup_{W:\,|W|\leq\delta}\compl_\Ncorner(W)<\frac{1}{3}c\inv\minLambda^{\Ncorner}.
  \]

  Let $X_N\subset\sing_{-N,0}$ denote the (finite) set of multiple points of $\sing_{-N,0}$.  We
  fix a small $\rho$ so that $\rho<\rho_N$, where $\{\rho_n\}$ is the sequence given by
  Lemma~\ref{l_portrait} and for any $0 <l\leq \Ncorner$, the diameter of every connected
  component of $\cm^l\nhbd_{z,\rho}$ is bounded above by $h/3$ for any $z\in X_N$.

  We choose $\delta$ so that for any u-curve $W$ with $|W|<\delta$:
  \begin{itemize}
  \item for any $0 <l\leq\Ncorner$, each component of $\cm^lW$ is shorter than $h/3$; we
    can ensure this by Lemma~\ref{l_maximalExpansion};
  \item for any $z,z'\in X_N$, $z'\not=z$, if $W\cap\nhbd_{z,\rho}\not = \emptyset$, then
    $W\cap\nhbd_{z',\rho}=\emptyset$ (by Lemma~\ref{l_portrait});
  \item if $W\cap\nhbd_{z,\rho}=\emptyset$ for each $z\in X_N$, then $\cm W$ has at most
    $2$ components (by Lemma~\ref{l_transversality});
  \item if $W\cap\nhbd_{z,\rho}\not=\emptyset$ for some $z\in X_N$, then $W$ is cut by
    $\sing_{-\Ncorner,0}$ in at most $\Xi\Ncorner$ components which are contained in a regular
    sector (by Lemma~\ref{l_linearBound} and Lemma~\ref{l_transversality}).
  \end{itemize}

  All components which belong to a nearly grazing sector are cut by
  $\exsing_{-\Ncorner,0}$ in H-components which are necessarily nearly grazing (by our
  assumptions on $\rho$ and $\deltaCorner$): as such they do not contribute to
  $\compl_\Ncorner(W)$.  On the other hand, a component belonging to a regular sector can
  be further split in H-components, but only at most one of them will be regular (the case
  $\Ncorner=1$ is trivial; the case $\Ncorner>1$ can be obtained by induction).  This
  concludes the proof of our lemma.
\end{proof}
Fix $z\in\cs$; let us call the straight billiard trajectory emanating from $z$ the
\emph{reference trajectory}; we denote by $x\in\partial\btable$ the starting configuration
point of the reference trajectory.  After --perhaps-- one or more improper collisions, the
reference trajectory will eventually properly collide with $\partial\btable$ at some point
$x'\in\partial\btable$.  By Remark~\ref{r_freePathImproper} we conclude that the number of
such improper collisions is uniformly bounded by $\tau_\text{max}/\tau_*$.  A trajectory
close to the reference trajectory could experience its first collision with different
boundary walls; any such possibility corresponds to a sector $\sect_i$ of order one (see
Figure~\ref{f_manyImproper}).
\begin{lem}\label{l_bounded1step}
  The number of sectors of order $1$ is uniformly bounded by $2(\tmax/\tau_*+1)$.
\end{lem}
\begin{proof}
  Note that we are only following the perturbed trajectories \emph{until the first
    collision}, so the question reduces to counting the possible ways in which this first
  collision can occur.  An improper collision can create at most two sectors, one
  corresponding to collisions occurring on the left (with respect to the selected
  orientation) of the improper collision, and one corresponding to collisions occurring on
  the right (as already noted earlier, only one is possible unless we have a tangential
  collision).  The remaining proper collision can create, for the same reason, at most two
  sectors.
\end{proof}
We can now give the
\begin{proof}[Proof of Lemma~\ref{l_grazing}]
  As in the proof of Lemma~\ref{l_corner}, let $h$ denote the height of the homogeneity
  strip $\sst_{k_0}$; choose $\delta$ so that, for any u-curve $W$ shorter than $\delta$:
  \begin{itemize}
  \item each component of $\cm W$ is shorter than $h/3$ (by
    Lemma~\ref{l_maximalExpansion});
  \item $W$ is cut in at most $2(\tmax/\tau_*+1)$ components (by
    Lemma~\ref{l_bounded1step}).
  \end{itemize}
  A nearly grazing component might be further split into H-components and will contribute
  with at most $\Const k_0\inv=\sum_{k\geq k_0}\Const k^{-2}$; the contribution of all
  nearly grazing components is thus bounded by $2 \Const (\tmax/\tau_*+1)k_0\inv$, which
  can be made arbitrarily small by taking $k_0$ to be large enough.
\end{proof}

We now proceed with the proof of Lemma~\ref{l_linearBound}.  Define the image sector
$\sect_i'=F\sect_i$ and let $z'_i$, which we call the \emph{center} of $\sect_i'$, be the image
of $z$ by the corresponding branch of the dynamics.
 \begin{figure}[!ht]
   \centering
   \includegraphics{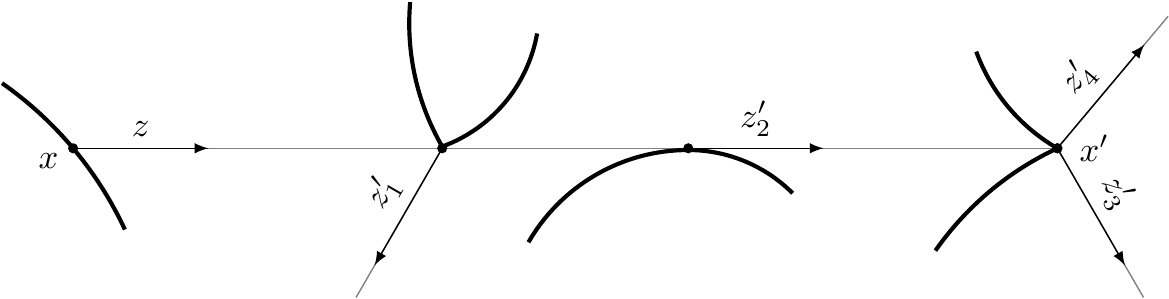}
   \caption{The straight billiard trajectory emanating from $z$ might encounter several
     improper collisions before properly hitting the boundary.  For each improper
     collision we have one or two possible new images $z'$. We stress that $z'$ are images
     of $z$ by a single iteration of $\cm$ (in particular notice that $z'_3$ and $z'_4$ in
     the picture will undergo several other immediate collisions before leaving the corner
     point $x'$).  Notice that $z'_4$ is only a theoretical possibility that we count, but
     it does not occur in reality.}
   \label{f_manyImproper}
 \end{figure}
 The next lemma refers to the key phenomenon which prevents complexity to grow fast in the
 class of billiards we are considering.  To understand it, it is worth to spend some time
 to explain the mechanism of complexity growth.  We are counting the number of
 \emph{sectors} into which a small neighborhood of $z$ is cut by singularities.
 \emph{Future time} singularities are \emph{stable} curves, so a sector $\sect$ bordered
 by two such singularities is either contained entirely in an active quadrant, or it
 contains an entire inactive quadrant.  On the other hand, the \emph{image} of such a
 sector is bordered by the \emph{images} of singularities, which are either the boundary
 of the phase space or \emph{past time} singularities, both being in (or on the boundary
 of) the \emph{unstable} quadrants.  As a result, the \emph{image} $\cm \sect$ is either
 contained entirely in an \emph{inactive} quadrant --then we call it an \emph{inactive
   sector}-- or it contains an entire \emph{active} quadrant and then we call it an
 \emph{active sector}.

 When considering higher iterates, we look at how $F(\sect)$ is further cut by
 future singularities.  Clearly, it can only be further cut if it is active.  To bound
 complexity, we first need to understand the number of active sectors as time evolves.

 If there are no corner points, one can make use of the continuity of the flow to see that
 the number of active sectors is always $2$ (resulting in linear complexity): there is
 simply no space for more.  However, in the presence of corner points, singular
 trajectories can branch, and after a branching there is in principle room for $4$ active
 sectors.  When we look at the possible collisions in detail, we will see that a corner
 collision --maybe somewhat surprisingly-- does not increase the number of active
 sectors, but a \emph{combination of corner and grazing collisions} is able to do that.
 In fact, a grazing collision is able to turn an inactive sector into an active one, which
 means that the number of active sectors \emph{can grow}, and complexity \emph{can be
   superlinear}.  We can (and did) avoid going into this in detail by making use of strong
 expansion near grazing collisions and counting \emph{regular} sectors only.  The key
 lemma that follows is about conservation of the number of \emph{regular active sectors}.
\begin{lem}\label{l_active}
  Let $\qdrn$ be an active quadrant of $\nhbd\ni z$ and $\{\sect_{\text{reg},j}\}$ denote regular
  sectors of order $1$ in which $\qdrn$ is subdivided by singularities.  Then, at most one of
  the regular image sectors $\{\sect'_{\text{reg},j}\}$ contains an active quadrant.
\end{lem}
Assuming the above lemma, we can now give the
\begin{proof}[Proof of Lemma~\ref{l_linearBound}]
  The proof follows from an argument that is similar to the one presented at the end of
  \cite[Section 4]{ChY}.  Let us fix a multiple point $z\in\sing_{-n,0}$ and let
  $\nhbd=\nhbd_{z,\rho_n}$; for an arbitrary sector $\sect\subset \nhbd$ we can define the quantity
  $\cplex_n(z)|\sect$, which is the number of regular sectors of order $n$ meeting at $z$
  which intersect non-trivially the sector $\sect$.  Let $\{\qdrn_\alpha\}$ denote the quadrants
  of $\nhbd$ (recall that there can be either $1$, $2$ or $4$); then necessarily
  \[
  \cplex_n(z)\leq\sum_\alpha\cplex_n(z)|\qdrn_\alpha.
  \]
  If $\qdrn$ is inactive then we have trivially $\cplex_n(z)|\qdrn=1$; therefore, it
  suffices to obtain a linear bound on $\cplex_n(z)|\qdrn$ if $\qdrn$ is active.  Let us
  fix an arbitrary active quadrant $\qdrn\subset \nhbd$ and let
  $\{\sect_j\}_{j\in\{1,\cdots,k\}}$ denote the regular sectors of order $1$ which
  intersect non-trivially $\qdrn$; by further cutting some of the $\sect_j$ along the
  vertical and horizontal axes, we can assume that all $\sect_j$ are indeed contained in
  $\qdrn$.  Notice that $k$ is uniformly bounded by some $K$ by
  Lemma~\ref{l_bounded1step}.  Recall that $\sect'_j=F\sect_j$ and $z'_j$ denotes the
  center of sector $\sect'_j$.  Then, by definition
  \[
  \cplex_n(z)|\qdrn = \sum_j\cplex_{n-1}(z'_j)|\sect'_j
  \]
  By Lemma~\ref{l_active}, only one of the $\sect'_j$s will be active, thus we have:
  \[
  \cplex_n(z)|\qdrn \leq K+ \cplex_{n-1}(z'_\text{active})|\sect'_\text{active}
  \]
  from which we can conclude by induction that $\cplex_n(z)|\qdrn<(K+2)n$, which proves our
  statement with $\Xi=2(K+2)+2$.
\end{proof}
We now come to the essence of this work and give the
\begin{proof}[Proof of Lemma~\ref{l_active}]
  Recall that $x\in\partial\btable$ and $x'\in\partial\btable$ denote respectively the
  starting and ending point of our reference trajectory; for ease of exposition, assume
  that the reference trajectory is horizontal and that $x'$ lies to the right of $x$ (as
  in Figure \ref{f_manyImproper}).  Additionally, assume the point $x'$ to be a corner
  point\footnote{ If $x'$ is a regular point, we can artificially break the corresponding
    wall at $x'$; this will only make the singularity set larger.}.  The reference
  trajectory will intersect a number of walls; we say that a wall is of \emph{type A}
  (resp. \emph{type B}) if it lies above (resp. below) the reference trajectory.  We say
  that a collision is of \emph{type A} (resp. of \emph{type B}) if it occurs with a wall
  of type A (resp. of type B).  By definition, all trajectories belonging to the same
  sector $\sect\subset \qdrn$ will have their first collision with the same wall, therefore their
  type depends on $\sect$ only: we thus naturally obtain the definition of sectors of
  \emph{type A} and sectors of \emph{type B}.

  To fix ideas we assume that $\qdrn$ is the NW quadrant, i.e. if $z=(x,\varphi)$, we only
  consider trajectories leaving from a left half-neighborhood of $x$ (that is, above the
  reference trajectory according to our choice for the orientation) with an angle slightly
  larger than $\varphi$ (see Figure~\ref{f_proof1}).  The proof for the SE quadrant
  follows by the same argument, performing a few simple modifications, and it is left to
  the reader.
  \begin{figure}[!h]
    \centering
    \centering
    \includegraphics{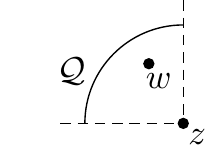}
    \includegraphics{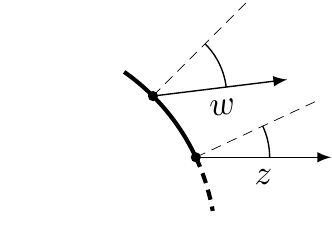}
    \caption{Phase points corresponding to the NW quadrant.}
    \label{f_proof1}
  \end{figure}

  We now proceed to prove the following items, which immediately imply our lemma:
  \begin{enumerate}[i)]
  \item there exists exactly one sector of type A;
  \item the image of any regular sector of type B does not contain an active quadrant.
  \end{enumerate}

  First, we claim that any trajectory emanating from $\qdrn$ and undergoing a collision of type A
  will necessarily hit the \emph{leftmost} $A$-wall; we denote this wall by $\Gamma_A$
  (see Figure~\ref{f_proof2}) and the corresponding sector by $\sect_A$.
  \begin{figure}[!h]
    \centering
    \includegraphics{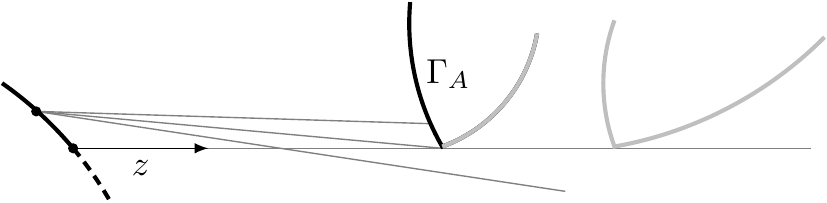}
    \caption{Type A collisions may only occur on the wall $\Gamma_A$; all other type A
      walls cannot be reached by any trajectory leaving from $\qdrn$.}
    \label{f_proof2}
  \end{figure}

  In fact, by elementary geometry considerations (see once again Figure~\ref{f_proof2}),
  any trajectory starting above the reference trajectory, and missing $\Gamma_A$ will
  cross the reference trajectory and will consequently undergo a type B
  collision, i.e. it will miss all other walls of type A.  This immediately implies that
  $\sect_A$ is the only sector of type A, which proves item i).  Notice moreover that the same
  considerations imply that trajectories leaving $\qdrn$ can only hit walls of type B which
  are either to the left or immediately to the right of $\Gamma_A$.

  Let us now consider collisions of type B: the situation is as in Figure~\ref{f_proof3}.
  \begin{figure}[!h]
    \centering
    \includegraphics{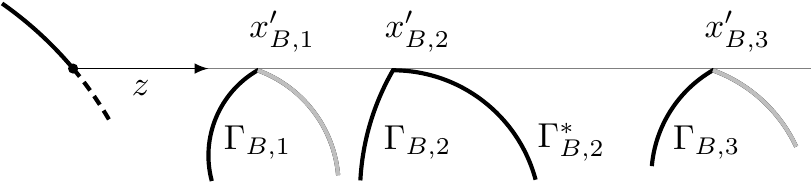}
    \caption{Several possibilities corresponding to type B collisions.}
    \label{f_proof3}
  \end{figure}
  As we mentioned earlier, each scatterer generates at most two new sectors; one
  corresponds to frontal collisions with the scatterer, the second (if present) to
  collisions with the back of the scatterer\footnote{If the collision is with a regular
    point, we can always artificially add a singularity which distinguishes front and back
    scatterings}.  Denote with $\Gamma_{B,j}$ walls facing $x$ (which give rise to frontal
  collisions) and with $\Gamma^*_{B,j}$ walls facing away from $x$ (which give rise to
  back collisions); we choose the index $j$ so that $\Gamma_{B,1}$ is the leftmost wall
  and respecting the left-to-right ordering of the scatterers.  Let us denote with
  $\sect_{B,j}$ and $\sect_{B,j}^*$ the corresponding sectors; a sample singularity portrait is
  depicted in Figure~\ref{f_quarter}.  Finally, denote by $x'_{B,j}$ the corner point
  corresponding to the $j$-th pair of walls, that is, the intersection of $\Gamma_{B,j}$
  (or $\Gamma_{B,j}^*$) with the reference trajectory.
  \begin{figure}[!h]
    \centering
    \includegraphics{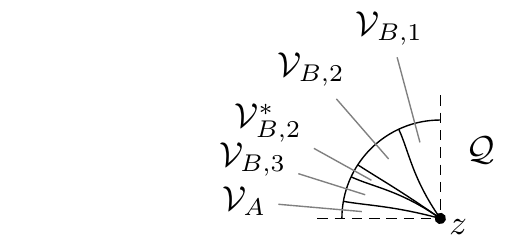}
    \caption{Typical singularity portrait of $\qdrn$: there is only one type A sector, and
      several type B sectors; sectors corresponding to collision with the back of a
      scatterers are nearly grazing.}
    \label{f_quarter}
  \end{figure}

  Our next remark is simple, but extremely important
  \begin{remark}
    Our choice of $\rho_n$ implies that we can have a back collision with $\Gamma_{B,j}^*$
    only if $\Gamma_{B,j}^*$ is tangent to the reference trajectory (see
    e.g. $\Gamma_{B,2}^*$ in Figure~\ref{f_proof3}).  In particular, all non-empty
    $\sect^*_{B,j}$ are \emph{necessarily} nearly grazing and thus they can be neglected.
  \end{remark}
  We now prove that every $\sect'_{B,j}$ is inactive, which, by the above remark, implies item
  ii).  Fix $j$ and let us denote $\sect_B=\sect_{B,j}$ for ease of exposition; let $\Gamma_B$ be
  the corresponding wall, $x'_B$ denote the corresponding corner point and $z'_B$ the
  corresponding phase point; similarly, let $x'_A$ be the corner point corresponding to
  $\Gamma_A$.  As a preliminary remark, notice that, by construction, $\sect'_B$ is contained
  in a left half-ball centered at $z'_B$; $\sect'_B$ is bounded by $3$ curves; we denote the
  two curves which join at $z'_B$ by $\cur_1$ and $\cur_2$ according to counterclockwise
  orientation (see Figure~\ref{f_proof4}).
  \begin{figure}[!h]
    \centering
    \includegraphics{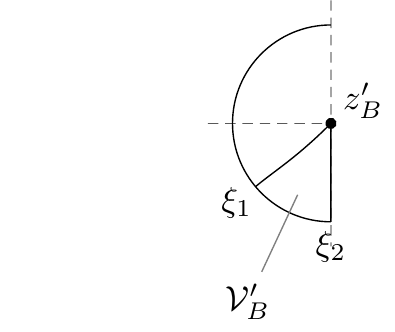}
    \includegraphics{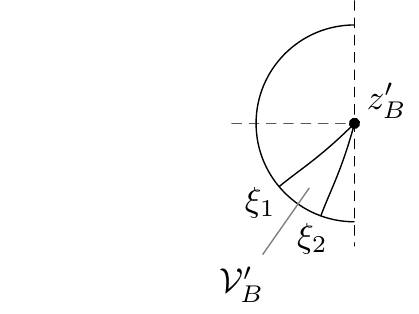}
    \includegraphics{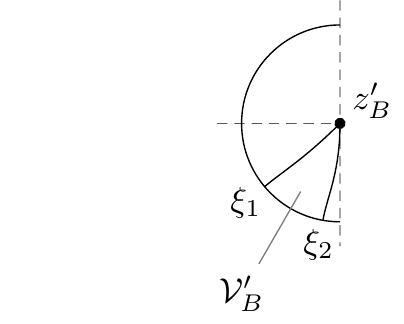}
    \caption{Possible portraits of $\sect'_B$}
    \label{f_proof4}
  \end{figure}

  We claim that $\cur_1$ is a u-curve: there are two possibilities (see once again
  Figure~\ref{f_proof3}):
  \begin{itemize}
  \item $j=1$, then $\cur_1$ is the image of the fan emanating from $x$, i.e. a
    u-curve;
  \item if $j>1$, then $\cur_1$ is the image of the fan emanating from $x'_{B,j-1}$, or
    the image of trajectories leaving tangentially $\Gamma^*_{B,j-1}$; in any of the two
    cases, we have a u-curve.
  \end{itemize}
  We claim that $\cur_2$ is either a u-curve or a vertical half-segment pointing downwards
  from $z'_B$.  Then our proof is complete: since u-curves are increasing, we conclude
  that $\sect'_B$ cannot contain an active quadrant.  There are three possibilities (see
  Figure~\ref{f_proof4}):
  \begin{enumerate}[(a)]
  \item $x'_A$ lies on the right of $x'_B$;
  \item $x'_A$ lies on the left of $x'_B$;
  \item $x'_A$ coincides with $x'_B$ (and thus with $x'$);
  \end{enumerate}
  Consider case (a): then there exist trajectories which leave $\qdrn$ and hit $x'_B$
  directly; then $\cur_2$ is the image of such trajectories, which corresponds to a
  vertical half-segment pointing downwards (this is depicted in the leftmost picture in
  Figure~\ref{f_proof4}).  In case (b), by our previous remark, $\Gamma_B$ has to be the
  wall immediately on the right of $\Gamma_A$: in this case there are no trajectories
  (excluding the reference one) leaving $\qdrn$ and hitting directly $x'_B$ and $\cur_2$ is
  either the image of the fan emanating from $x'_A$, or the image of trajectories leaving
  tangentially $\Gamma_A$ (if $\Gamma_A$ intersect the reference wall with a tangency).
  In any case $\cur_2$ turns out to be a u-curve (see the central picture in
  Figure~\ref{f_proof4}). In case (c) there are two possibilities, either $\Gamma_A$ is
  tangent to the reference trajectory, or not; in the former case we conclude by the
  argument described in case (b) (and this is depicted by the rightmost picture in
  Figure~\ref{f_proof4}). The latter case follows from the argument presented in case (a)
  (and corresponds once again to the leftmost picture in Figure~\ref{f_proof4}).  This
  proves our claim, and concludes the proof of our main lemma.
\end{proof}
\bibliographystyle{abbrv} \bibliography{bil}
\end{document}